\documentclass[a4paper, oneside, reqno,12pt]{amsart}
\usepackage[english]{babel}
\usepackage[T1]{fontenc}

\usepackage{amsmath,amscd,amsthm,amsfonts,latexsym,epsfig}

\usepackage[left = 2.5 cm, right = 2.5 cm, top = 2.5 cm, bottom = 2.5 cm]{geometry}

\usepackage{graphicx}
\usepackage{color}
\usepackage{subfigure}

\usepackage{hyperref}
\hypersetup{
	colorlinks = true,
	linkcolor = red,
	filecolor = blue,
	urlcolor = blue,
	citecolor = blue,
    pdftitle = {liouville-CDOU}
}

\theoremstyle{plain}
\newtheorem{theorem}                 {Theorem}      [section]
\newtheorem{proposition}  [theorem]  {Proposition}
\newtheorem{corollary}    [theorem]  {Corollary}

\newtheorem*{theorem*}{Theorem}
\theoremstyle{definition}
\newtheorem{example}      [theorem]  {Example}
\newtheorem{remark}       [theorem]  {Remark}
\newtheorem{definition}   [theorem]  {Definition}

\renewcommand{\Re}{\mathcal Re\,}
\renewcommand{\Im}{\mathcal Im\,}
\renewcommand{\r}{\mathbb R}
\newcommand{\C}{\mathbb C}
\newcommand{\h}{\mathbb H}

\renewcommand{\L}{\mathbb{L}}
\newcommand{\K}{\mathbb{K}}

\begin{document}
\title[A Liouville--Weierstrass correspondence for minimal surfaces in \texorpdfstring{$\mathbb{L}^3$}{L3}]{A Liouville--Weierstrass correspondence for Spacelike and Timelike Minimal Surfaces in \texorpdfstring{$\mathbb{L}^3$}{L3}}

\author[A. A. Cintra]{Adriana A. Cintra}
     \address{Instituto de Matemática e Estatística, Universidade Federal de Goiás, Campus Samambaia,
Avenida Esperança, s/n. 74690-900, Goiânia, GO, Brazil}
\email{adriana.cintra@ufg.br}

\author[I. Domingos]{Iury Domingos}
\address{Universidade Federal de Alagoas\\
Av. Manoel Severino Barbosa S/N,
57309-005 Arapiraca - AL, Brazil}
\email{iury.domingos@arapiraca.ufal.br}

\author[I. I. Onnis]{Irene I. Onnis}
\address{Dipartimento di Matematica e Informatica, Università degli Studi di Cagliari, Via Ospedale 72,
09124 Cagliari, Italy}
\email{irenei.onnis@unica.it}


\thanks{The authors were supported by a grant of Fondazione di Sardegna. I.~Domingos and  I.I.~Onnis were partially supported by the Brazilian National Council for Scientific and Technological Development (CNPq), grant no.~409513/2023-7. I.I.~Onnis was supported by GNSAGA-INdAM and also by the Thematic Project: Topologia Algebrica,  Geométrica e Diferencial,  Fapesp process number 2022/16455-6. }


\begin{abstract}
We investigate a correspondence between the solutions $\lambda(x,y)$ of the
Liouville equation
\[
\Delta \lambda = -\varepsilon e^{-4\lambda},
\]
and the Weierstrass representations of spacelike ($\varepsilon = 1$) and
timelike ($\varepsilon = -1$) minimal surfaces with diagonalizable
Weingarten map in the three-dimensional Lorentz--Minkowski space
$\mathbb{L}^3$.
Using complex and paracomplex analysis, we provide a unified treatment of both causal types.
We study the action of pseudo-isometries of $\mathbb{L}^3$ on minimal
surfaces via M\"obius-type transformations, establishing a
correspondence between these transformations and rotations in the
special orthochronous Lorentz group.
Furthermore, we show how local solutions of the Liouville equation
determine the Gauss map and the associated Weierstrass data.
Finally, we present explicit examples of spacelike and timelike minimal
surfaces in $\mathbb{L}^3$ arising from solutions of the Liouville
equation.
\end{abstract}

\maketitle

\section{Introduction}

Let $\Delta$ denote the standard Laplacian on $\mathbb{R}^2$. The classical Liouville equation
\begin{equation}\label{liouvilleeq}
\Delta \mu + K\, e^{2\mu} = 0, \qquad K \in \mathbb{R},
\end{equation}
is a nonlinear partial differential equation that dates back not only to the original work of Liouville~\cite{liouville}, but also to subsequent contributions by Picard~\cite{Pic, Pic1} and Poincaré~\cite{Poi}.

This equation has deep connections to complex analysis because it admits a holomorphic resolution, due to Liouville. Namely, any solution defined on a simply connected domain $\Omega \subset \mathbb{R}^{2} \equiv \mathbb{C}$ can be written in the form
\begin{equation}\label{mu}
\mu(x,y) = \log\!\left( \frac{2\, |g'(z)|}{1 + K\, |g(z)|^{2}} \right), 
\qquad z = x + i y,
\end{equation}
where $g$ is a locally univalent meromorphic function on $\Omega$ (holomorphic with $1 + K\, |g|^{2} > 0$ if $K \leq 0$). Conversely, if $g$ is a locally univalent meromorphic function on $\Omega$ (holomorphic with $1 + K\, |g|^{2} > 0$ if $K \leq 0$), then \eqref{mu} is a solution to \eqref{liouvilleeq} in $\Omega$.

Later, in \cite{bl, bln}, the authors proved that the function $g$, known as the \emph{developing map} of the solution $\mu$, is uniquely determined up to a Möbius transformation of the form
\begin{equation}\label{TM}
T_{ab}(g)=\frac{a\, g - b}{K\, \overline{b}\, g + \overline{a}}, 
\qquad |a|^{2} + K\, |b|^{2} = 1.
\end{equation}
These transformations are isometries of the $2$-dimensional space form of constant curvature~$K$.

In general, Liouville's representation formula does not hold on domains that are not simply connected. For example, consider the function
\[
\mu(x,y) = -\frac{1}{2}\log\!\bigl(|z|\,(1+K\,|z|)^2\bigr),
\]
which solves equation~\eqref{liouvilleeq} in the punctured disk $D^*=\{z\in\mathbb{C}:0<|z|<1\}$, with an isolated singularity at the origin. However, it is easy to verify that this solution is given, via Liouville's formula, by the multivalued analytic function $g(z)=z^{1/2}$, rather than by a single-valued analytic function on $D^*$.

In \cite{CW}, Chou and Wan extended the representation formula for solutions of \eqref{liouvilleeq} to a punctured disk in terms of multivalued meromorphic functions. Moreover, in \cite{bhl}, the authors proved a global version of Liouville's formula on arbitrary planar domains and, as an application, recovered the representation formula obtained by Chou and Wan. More recently, further extensions of the problem \eqref{liouvilleeq} to other domains (such as disks, half-planes, or annuli) have been investigated in the presence of Neumann boundary conditions. For instance, the half-plane case was studied by Gálvez and Mira \cite{Galvez-Mira} (see also \cite{Galvez-Jimenez-Mira}), while the annular case in $\mathbb{R}^2$, under suitable Neumann boundary conditions on each boundary component, was considered by Jiménez \cite{Jimenez2012}, among others.

Historically, an important connection between equation~\eqref{liouvilleeq} and differential geometry was established by Monge in his {\em Applications d'Analyse \`a la G\'eom\'etrie} (1849), where he proved that, in dimension two, its solutions $\mu$ give rise to the conformal factor $e^{2\mu}$ that transforms the flat metric $dx^2 + dy^2$ of the Euclidean plane into a metric of constant Gaussian curvature equal to $K$.

L.~Bianchi, in \cite{bianchi}, clarified the link between the Liouville equation and the theory of minimal surfaces. If $d\sigma^2$ is a prescribed Riemannian metric on a surface $\Sigma$ and $K_g<0$ denotes its Gaussian curvature then, as proved by Ricci-Curbastro in \cite{ricci}, $\Sigma$ can be locally realized as a minimal surface in $\mathbb{R}^3$ with induced metric $d\sigma^2$ if and only if the associated metric $\sqrt{-K_g}\, d\sigma^2$ is flat, or, equivalently, if the following condition holds:
\begin{equation}\label{eqricci}
\Delta_{d\sigma^2}\log (-K_g)=4K_g,
\end{equation}
where $\Delta_{d\sigma^2}$ denotes the Laplace--Beltrami operator associated with $d\sigma^2$. Assuming that $d\sigma^2=\lambda(x,y)\, (dx^2+dy^2)$, since $K_g=-\lambda^{-2}$ and $\Delta_{d\sigma^2}=\Delta/\lambda$, it is easy to see that equation~\eqref{eqricci} can be rewritten in the form
\[
\Delta \log\lambda=\frac{2}{\lambda}.
\]
Hence, setting $\lambda=e^{-2\mu}$, we obtain the Liouville equation~\eqref{liouvilleeq} with constant $K=1$. Bianchi refers to equation~\eqref{eqricci} as the {\em Ricci condition}. More generally, sufficient conditions for a Riemannian metric $d\sigma^2$ to be locally realized as the induced metric of constant mean curvature $H$ surface in $\mathbb{R}^3$ are that the Gaussian curvature $K_g$ satisfies $H^2-K_g>0$ and that the metric $\sqrt{H^2-K_g}\, d\sigma^2$ be flat.

In \cite{bln}, the authors study equation~\eqref{liouvilleeq} in dimension two and review Liouville’s formula expressing a solution $\lambda$ in terms of the Weierstrass representation of a minimal surface in $\mathbb{R}^3$. They show that a minimal surface without umbilic points admits a local Weierstrass representation with data $(dz/g'(z), g(z))$, where $g$ is the meromorphic function representing the surface normal map, which produces a solution of the Liouville equation~\eqref{liouvilleeq}. Conversely, a solution $\lambda$ of the Liouville equation~\eqref{liouvilleeq}, defined on a simply connected flat region, determines, up to a rigid motion of $\mathbb{R}^3$, a minimal immersion without umbilic points. The immersion is completely described in terms of a meromorphic function $g$, which is uniquely determined up to a Möbius transformation~\eqref{TM}, with $K=1$. In \cite{bhl}, when $K=-1$, the authors extend the representation given in \cite{bln}. In this case, the correspondence is between solutions of the Liouville equation~\eqref{liouvilleeq} and minimal spacelike surfaces in the Lorentz--Minkowski space $\mathbb{L}^3$.

The aim of this paper is to study the correspondence between solutions
$\lambda(x,y)$ of Liouville equation
\begin{equation}\label{lml}
\Delta \lambda = -\varepsilon\, e^{-4\lambda},
\end{equation}
where $\Delta$ denotes the Laplace--Beltrami operator of the metric
$e^{2\lambda}(dx^2+\varepsilon\, dy^2)$, and minimal surfaces in the Lorentz--Minkowski space $\mathbb{L}^3$.
We consider both spacelike $(\varepsilon=1)$ and timelike $(\varepsilon=-1)$
surfaces with diagonalizable Weingarten map.
In the space $\mathbb{L}^3$, that is, $\mathbb{R}^3$ endowed with the Lorentzian metric
\[
dx_1^2 + dx_2^2 - dx_3^2,
\]
 a Weierstrass-type representation theorem was proved by Kobayashi for spacelike minimal immersions (see \cite{Kob}), and by Konderak for the case of timelike minimal surfaces (see \cite{konderak}). The results of Konderak were later generalized by Lawn in \cite{l}. More recently, these results were extended to immersed minimal surfaces in Riemannian and Lorentzian three-dimensional manifolds by Lira et al.\ (see \cite{Liramm}).

The paper is organized as follows. In Section~\ref{three} we recall the analytic tools required for the study of minimal surfaces in the Lorentz--Minkowski space $\mathbb{L}^3$. In particular, we review the algebra of Lorentz (or paracomplex) numbers and the corresponding notion of differentiability, which plays a role analogous to complex calculus in the timelike setting. This section also presents a unified Weierstrass-type representation for spacelike and timelike minimal surfaces in $\mathbb{L}^3$, including the classical formulation due to Kobayashi and the paracomplex approach introduced by Konderak. 

In Section~\ref{four} we examine pseudo-isometries of $\mathbb{L}^3$ and their relation with $\mathbb{K}$-bilinear (M\"obius-type) transformations. The main result of this section is Theorem~\ref{teo2}, which establishes a correspondence between these bilinear transformations and rotations in the special orthochronous Lorentz group, and describes how such transformations act on the Weierstrass data. 

Section~\ref{sfive} is devoted to the relationship between minimal surfaces in $\mathbb{L}^3$ and the Liouville equation. We show that any local solution $\lambda$ of the Liouville equation determines a (para)meromorphic function $g$, unique up to $\mathbb{K}$-bilinear transformations, such that
\[
e^\lambda=\dfrac{|1 -\varepsilon\, g\overline{g}|}{2\sqrt{g'\overline{g'}}}.
\]
This result is stated in Theorem~\ref{teo3} and provides a correspondence between solutions of the Liouville equation and Gauss maps of spacelike or timelike minimal surfaces. 

Finally, in Section~\ref{ssix} we present explicit examples of spacelike and timelike minimal surfaces in $\mathbb{L}^3$, obtained from solutions of the Liouville equation. These examples include classical models as well as additional families arising from the proposed framework.

\subsection*{Acknowledgements} I.~Domingos and I.~I.~Onnis also wish to express their sincere gratitude to the Abdus Salam International Centre for Theoretical Physics (ICTP) for its warm hospitality during the Research in Pairs Programme in June 2025, where this work was completed.

\section{The Weierstrass representation formula in \texorpdfstring{$\mathbb{L}^3$}{L3} and the Lorentz Numbers}\label{three}

In \cite{konderak}, the author uses paracomplex analysis to prove a Weierstrass representation formula for timelike minimal surfaces immersed in the space $\L^3$.
We recall that the algebra of {\em paracomplex (or Lorentz) numbers} is the algebra $$\L = \{a + \tau\, b\;|\; a,b \in \r\},$$ where $\tau$ is an imaginary
unit with $\tau^2 = 1$. The two internal operations are the usual ones. We define the conjugation in $\L$ as $\overline{a + \tau\, b} := a - \tau\, b$ and the
$\L$-norm of $z = a + \tau\, b \in \L$ is defined by $|z| = |z\,\overline{z}|^{\frac{1}{2}} = |a^2 - b^2|^{\frac{1}{2}}.$
The algebra $\L$ contains the set of zero divisors $\mathcal{C} = \{a\pm \tau\,a \,:\, a~\neq~0\}$.
If $z\notin \mathcal{C}\cup\{0\}$, then it is invertible with inverse $\displaystyle z^{-1} = \bar{z}/(z\bar{z})$.
We observe that $\L$ is isomorphic to the algebra $\r \oplus \r$ via the map
$
\Phi(a + \tau\, b) = (a + b, a - b)
$
and the inverse of this isomorphism is given by $\Phi^{-1}(a,b)=(1/2)\,[(a+b)+\tau (a-b)]$.
Also, $\L$ can be canonically endowed with an indefinite metric by
$$\langle z,w\rangle= \Re\, (z\,\bar{w}), \qquad z,w\in \L.$$

\subsection{Differentiability
over paracomplex numbers}\label{algL}

In the following, we introduce the notion of differentiability
over Lorentz numbers and some properties (see \cite{konderak2} for more details).
\begin{definition}
Let $\Omega\subseteq \L$ be an open set\footnote{ The set $\L$ has a natural topology since it is a two-dimensional real vector space.} and $z_0 \in \Omega$.
The $\L$-derivative of a function $f:\Omega\rightarrow\L$ at $z_0$ is defined by
\begin{equation}\label{eq:A11}
f'(z_0) := \lim_{\substack{z \to z_0 \\ z - z_0 \in \mathbb{L} \setminus \mathcal{C} \cup \{0\}}}
\frac{f(z) - f(z_0)}{z - z_0}.
\end{equation}
if the limit exists. If $f'(z_0)$ exists, we say that $f$ is $\L$-differentiable
at $z_0$. When $f$ is $\L$-differentiable at all points of $\Omega$ we say that $f$ is $\L$-holomorphic in $\Omega$.
\end{definition}

\begin{remark}
The condition of $\L$-differentiability is much less restrictive than the usual complex differentiability.
For example, $\L$-differentiability at $z_0$ does not imply continuity at $z_0$. However, $\L$-differentiability
in an open set $\Omega \subset \L$ implies usual differentiability in $\Omega$. Also, we point out that there exist $\L$-differentiable functions of any class of usual differentiability (see \cite{konderak2}).
\end{remark}
 Introducing the paracomplex operators,
\begin{equation}\label{dpc}
\frac{\partial}{\partial z} = \frac{1}{2}\Big(\frac{\partial}{\partial x} + \tau\frac{\partial}{\partial y}\Big),\qquad
\frac{\partial}{\partial \bar{z}} = \frac{1}{2}\Big(\frac{\partial}{\partial x} - \tau\frac{\partial}{\partial y}\Big),\nonumber
\end{equation}
where $z = x + \tau\, y$, we can give a necessary and sufficient condition for the $\L$-differentiability of a function $f$ in some open set.
\begin{theorem} Let $a,b:\Omega\to\L$ be $C^1$ functions in an open set $\Omega\subset\L$. Then the function $f(x,y) = a(x,y) +\tau\, b(x,y)$, $x + \tau y\in\Omega$, is $\L$-holomorphic in $\Omega$ if and only if \begin{equation}\label{eqldif}
\displaystyle\frac{\partial f}{\partial \bar{z}} = 0
\end{equation}
is satisfied at all points of $\Omega$.
\end{theorem}
Observe that the condition~\eqref{eqldif} is equivalent to
the para-Cauchy-Riemann equations
$$
a_x = b_y,\quad
a_y =b_x
$$
and, in this case, we have that
$$\begin{aligned}f'(z)&=a_x(x,y)+\tau\,b_x(x,y)=b_y(x,y)+\tau\,a_y(x,y)\\&=\frac{1}{2}\Big(\frac{\partial}{\partial x} + \tau\frac{\partial}{\partial y}\Big)(f).\end{aligned}$$
\begin{remark}
If $f$ is a  $\L$-differentiable function, from the  para-Cauchy-Riemann equations we have that
\begin{equation}\label{deri}
f_z=2 (\Re f)_z=2\tau (\Im f)_z.
\end{equation}
\end{remark}

\subsection{Some elementary functions over the Lorentz numbers}\label{formulas}
In the following, we shall write functions of the Lorentz variable $z=x +\tau y$ in the ``sans serif style'' to distinguish them from the corresponding classical complex functions whose domain is contained in $\C$. In \cite{konderak2} the authors define the exponential function
$$\mathsf{exp}(z):=e^x\,(\cosh y +\tau\,\sinh y),\quad z\in\L.$$ Putting $x=0$, we obtain
$$\mathsf{exp}(\tau\,y)=\cosh y+\tau\,\sinh y, \quad \mathsf{exp}(-\tau\,y)=\cosh y-\tau\,\sinh y$$ and
$$
\cosh y=\frac{\mathsf{exp}(\tau y)+\mathsf{exp}(-\tau y)}{2},\quad \sinh y=\frac{\mathsf{exp}(\tau y)-\mathsf{exp}(-\tau y)}{2\tau}.
$$
These expressions may be used to continue hyperbolic cosine and sine as $\L$-holomorphic functions in
the whole set $\L$ setting
$$
\mathsf{cosh} (z):=\frac{\mathsf{exp}(\tau z)+\mathsf{exp}(-\tau z)}{2},\quad
\mathsf{sinh} (z):=\frac{\mathsf{exp}(\tau z)-\mathsf{exp}(-\tau z)}{2\tau },
$$
for all $z\in \L$. It is easy to check the following formulas
\begin{equation}
\begin{aligned}
\mathsf{cosh} (z)&=\cosh x\,\cosh y+\tau\, \sinh x\,\sinh y,\\
\mathsf{sinh} (z)&=\sinh x\,\cosh y+\tau\,\cosh x\,\sinh y.
\end{aligned}
\end{equation}
We observe that $$\mathsf{exp}(\tau\,z)=\mathsf{cosh} (z)+\tau\,\mathsf{sinh} (z)$$ and $\mathsf{sinh}' (z)=\mathsf{cosh} (z),$ $\mathsf{cosh}'(z)=\mathsf{sinh} (z)$, for all $z\in\L.$
Also,
\begin{equation}\label{proprieta}
\mathsf{cosh} (\tau z)=\mathsf{cosh} (z), \quad \mathsf{sinh} (\tau z)=\tau\,\mathsf{sinh} (z),\quad z\in\L.
\end{equation}
Extending \eqref{proprieta} to circular trigonometric functions and applying the usual angle addition formulas, we define
\begin{equation}
\begin{aligned}
\mathsf{sin} (z)&:=\sin x\,\cos y+\tau\,\cos x\,\sin y,\\
\mathsf{cos} (z)&:=\cos x\,\cos y-\tau\, \sin x\,\sin y,\quad z\in\L.
\end{aligned}
\end{equation}
These functions are $\L$-differentiable in $\L$ and they satisfy the same differentiation formulas which hold
for real and complex variables.

\subsection{The Weierstrass representation formula in \texorpdfstring{$\mathbb{L}^3$}{L3}}\label{weier}
We denote by $\K$ either the complex numbers $\C$ or the paracomplex numbers $\L$, and by $\Omega\subset \K$  an open set. Given  a smooth immersion $\psi:\Omega\subset \K \rightarrow \L^3$, we endow $\Omega$ with the induced metric $ds^{2}=\psi^{*} \langle\,,\,\rangle$, which makes $\psi$ an isometric immersion. We say that $\psi$ is {\it spacelike} if the induced metric  $ds^{2}$ is a Riemannian metric, and that
$\psi$ is {\it timelike} if the induced metric $ds^{2}$ is a Lorentzian metric.

We observe that in the Lorentzian case, we can endow $\Omega$ with paracomplex isothermic coordinates and, as in the Riemannian case, they are locally described by paracomplex isothermic charts with conformal changes of coordinates (see \cite{tila}). Let $z=x+i\,y$ (respectively, $z=x+\tau\,y$) be a complex (respectively, paracomplex) isothermal coordinate in $\Omega$,
so that
$$ds^2\Big(\dfrac{\partial}{\partial x},\dfrac{\partial}{\partial x}\Big)=\varepsilon \,ds^2\Big(\dfrac{\partial}{\partial y},\dfrac{\partial}{\partial y}\Big),\qquad ds^2\Big(\dfrac{\partial}{\partial x},\dfrac{\partial}{\partial y}\Big)=0,$$
where $\varepsilon=1$ (respectively, $\varepsilon=-1$).
It follows that
there exists a function $\lambda:\Omega\to\r$ such that the induced metric is given by $ds^{2}=e^{2\lambda}\,(dx^{2}+\varepsilon\, dy^{2})$, where
 \begin{equation}\label{lambda}
e^{2\lambda}=\frac{\langle\psi_x,\psi_x\rangle+\varepsilon \langle\psi_y,\psi_y\rangle}{2}=2\,\langle\psi_z,\psi_{\bar{z}}\rangle.
\end{equation}
Observe that the Beltrami-Laplace operator (with respect to $ds^{2}$) is given by:
\begin{equation}\label{beltrami}
\Delta=\frac{1}{e^{2\lambda}}\,\Big(\dfrac{\partial}{\partial x}\,\dfrac{\partial}{\partial x}+\varepsilon\, \dfrac{\partial}{\partial y}\,\dfrac{\partial}{\partial y}\Big)=\frac{4}{e^{2\lambda}}\,\dfrac{\partial}{\partial \bar{z}}\,\dfrac{\partial}{\partial z}.
\end{equation}
Also, denoting by $N$ the unit normal vector field along $\psi$, which is timelike (respectively, spacelike) if $\psi$ is a spacelike (respectively, timelike) immersion (i.e. $\langle N,N \rangle=-\varepsilon$), it follows that
$\triangle \psi=2\overrightarrow{H},$
where $\overrightarrow{H}=H\,N$ is the mean curvature vector of $\psi$ (i.e. $H$ is half  of the trace of the second fundamental form with respect
to the first fundamental ).In particular, the immersion $\psi$ is minimal (i.e. $H\equiv 0$) if and only if the coordinate functions $\psi_j$, $j=1,2,3$, are harmonic functions, or equivalently  $(\partial \psi_j/\partial z)$,  $j=1,2,3$, are $\K$-differentiable.

In the following, we state the Weierstrass representation type theorem for spacelike  (respectively, timelike) minimal immersions in $\L^{3}$, that was proved by Kobayashi in \cite{Kob} (respectively, by Konderak in \cite{konderak}), in a unified version.

\begin{theorem}[Weierstrass representation]\label{teo1}
Let $\psi:\Omega\subset \K \rightarrow \L^3$ be a smooth conformal minimal  spacelike  (respectively, timelike) immersion. Then, the (para)complex tangent vector defined by
\begin{equation}
\phi(z) := \frac{\partial \psi}{\partial z}\bigg{|}_{\psi(z)}=\sum_{i=1}^3  \phi_i\,\frac{
\partial}{\partial x_i},\nonumber
\end{equation}
 satisfy the following conditions:
\begin{itemize}
\item[(i)]$ \phi_1\,\overline{\phi_1} + \phi_2\,\overline{\phi_2}- \phi_3\,\overline{\phi_3} \neq 0,$
\item[(ii)] $\phi_1^{2}+\phi_2^{2}-\phi_3^{2}= 0,$
\item[(iii)] $\displaystyle\frac{\partial\phi_j}{\partial\bar{z}} = 0, \; j=1,2,3$,
\end{itemize}
where $\dfrac{\partial}{\partial z}$ and $ \dfrac{\partial}{\partial\bar{z}}$ are the  (para)complex operators.

Conversely, if $\Omega\subset\K$ is a simply connected domain and $\phi_j:\Omega\to \K$, $j = 1,2,3$, are (para)complex functions satisfying the
conditions above, then the map
\begin{equation}
\psi= 2\,\Re\int_{z_0}^{z} \phi\, dz,
\end{equation}
is a well-defined conformal spacelike (respectively, timelike) minimal immersion in $\L^{3}$ (here, $z_0$ is an arbitrary fixed point of $\Omega$ and the integral
is along any curve joining $z_0$ to $z$)\footnote{The $\K$-differentiability ensures that the  $1$-forms $\phi_ j\, dz$, $j=1,2,3$, do not have real periods in $\Omega$.}.
\end{theorem}

\begin{remark}
The first condition of Theorem~\ref{teo1} ensures that $\psi$ is an immersion (see \eqref{lambda}), the second one that $\psi$ is conformal and the third one that $\psi$ is minimal.
\end{remark}

Another version of the Weierstrass representation formula can be obtained considering  a \linebreak $\K$-holomorphic 1-form locally it follows that defined by $f(z)dz$ and a meromorphic function $g$ with the same zeros and poles, such that $fg^2$ is nonzero and $\K$-holomorphic at the eventual poles of $g$. The Weierstrass data is given by:
\begin{equation}\label{fg}
(f,g)=\left\{\begin{array}{cc}
     \bigg(2(\phi_1 - i \phi_2), \dfrac{\phi_3}{-\phi_1 + i\phi_2} \bigg), \quad\mbox{if} \quad \K=\C, \vspace{0.5cm}\\
     
      \bigg(2(\phi_2 + \tau \phi_3), \dfrac{\phi_1}{\phi_2 + \tau\phi_3} \bigg),\quad \mbox{if} \quad \K=\L,
\end{array}
\right.  
\end{equation}
with $g\overline{g}\neq \varepsilon$ and $f\overline{f}>0$. So, it follows that
\begin{equation}\label{RW}
2\psi_z = \left\{\begin{aligned}
     &\Big(\dfrac{1}{2}f\,(1 + g^2),\dfrac{i}{2}f\,(1 - g^2),-fg\Big), \quad\mbox{if} \quad \K=\C,
   \\
     &\Big(fg, \dfrac{1}{2}f\,(1 - g^2),\dfrac{\tau}{2}f\,(1 + g^2)\Big),\qquad \mbox{if} \quad \K=\L.
\end{aligned}
\right.  
\end{equation}

Now, we consider the hyperboloid in $\L^3$ defined by:
$$\mathcal{H}^2_\varepsilon =\{(u,v,w)\in\L^3\,|\, u^2 + v^2 - w^2 = - \varepsilon\}.$$

If $\varepsilon = 1$, let us refer to the point $P_N=(0, 0, 1)$ as the north
pole. The stereographic projection $\pi:\mathcal{H}^2_1-\{P_N\} \rightarrow \C$ from the north pole, that maps each point
$(u, v,w)$ of the two-sheeted hyperboloid $\mathcal{H}^2_1$ distinct from $P_N$ to the
point $z =x+iy$ of the equatorial plane, is given by
\begin{equation}\label{PEspacelike}
  \pi(u,v,w) = \dfrac{u + i\,v}{w - 1}.
\end{equation}

  If $\varepsilon = -1$, let us refer to the point $P_S=(-1, 0, 0)$ as the south
pole. The stereographic projection $\pi:\mathcal{H}_{-1}^2-\{P_S\} \rightarrow \L$ from the south pole, that maps each point
$(u, v,w)$ of one-sheeted hyperboloid $\mathcal{H}_{-1}^2$ distinct from $P_S$ to the
point $z =x+\tau y$ of the plane-$yz$, is given by
\begin{equation}\label{PEtimelike}
  \pi(u,v,w) = \dfrac{-v + \tau\, w}{u +  1}.
\end{equation}

\section{Pseudo-isometries in \texorpdfstring{$\mathbb{L}^3$}{L3} and \texorpdfstring{$\mathbb{K}$}{K}-bilinear transformations}\label{four}
The concept of a {\em Möbius (or bilinear) transformation} defined for complex variables can be extended to the case of paracomplex variables (see \cite{catoni}). Therefore, in this paper, we refer to a {\em $\K$-bilinear transformation} as a mapping $T:\K\rightarrow \K$ of the form 
$$
T(z)=\dfrac{a\, z+ b}{c\, z+d},  
$$
where $a,b,c,d$ are (para)complex numbers that must satisfy the conditions   $|ad-bc|\neq 0$ and $c\neq 0$.
In this section, we prove that any pseudo-isometry of $\mathbb L^3$ can be written as a $\K$-bilinear transformation 
\begin{equation}\label{mobius}
T_{ab}(z)=\dfrac{a\,z+\varepsilon\, b}{\bar b\, z+\bar a},
\end{equation}
where $a\bar a-\varepsilon\, b\bar b =1$ and $\bar b z+\bar a\notin \mathcal{C}$. 
\begin{definition}[Pseudo-Euclidean isometry \cite{Alexandre, RL}]
 A pseudo-Euclidean isometry in $\mathbb{L}^3$ is a map $\Phi:\mathbb{L}^3\rightarrow\mathbb{L}^3$ such that 
$$\langle\Phi(x)-\Phi(y),\Phi(x)-\Phi(y)\rangle=\langle x-y,x-y \rangle,$$
 for any vectors $x,y\in \mathbb{L}^3$. We denote the set of all pseudo-Euclidean isometries by $E_1(3,\mathbb{R})$.
\end{definition}

 \begin{definition}[Pseudo-orthogonal transformation \cite{Alexandre,RL}] A pseudo-orthogonal transformation in $\mathbb{L}^3$ is a linear map $\Psi:\mathbb{L}^3\rightarrow\mathbb{L}^3$ such that 
  $$\langle\Psi(x),\Psi(y)\rangle=\langle x,y \rangle,$$
 for any vectors $x,y\in \mathbb{L}^3$. We denote the set of all pseudo-orthogonal transformations by $O_1(3,\mathbb{R})$.
 \end{definition}

The set $O_1(3,\mathbb{R})$ is a group if endowed with the usual matrix multiplication, and it is 
called the {\it Lorentz group}. The group 
$$SO_1(3,\r)=\{A \in O_1(3,\r)\;|\; \textrm{det} A=1\}$$ 
is called the {\it special Lorentz group}.
 \begin{theorem}
 Given $\Phi\in E_1(3,\mathbb{R})$, there exists a unique $a\in \mathbb{L}^3$ and $\Psi\in O_1(3,\mathbb{R})$ such that $\Phi=t_a\circ \Psi$, where $t_a$ is a translation.
 \end{theorem}

\begin{definition}
We say that $A\in O_1(3,\r)$ preserves the timelike orientation if given a future-directed orthonormal base $B$, then the base obtained by $B'=A\cdot B$ is also future-directed.  We define the {\it orthochronous group} by
$$O^{+}_1(3,\r) = \{A \in O_1(3,\r) \;|\; A\; \text{preserves the timelike orientation}\}.$$
\end{definition}
We also have the next characterization of $O^{+}_1(3,\r)$:
$$ A \in O^{+}_1(3,\r) \quad \mbox{if and only if} \quad a_{33} > 0.$$

We define the {\it special orthochronous Lorentz group} as the set
$$O^{++}_1(3,\r) = SO_1(3,\r)\cap O^{+}_1(3,\r) = \{A \in O_1(3,\r)\;|\; \textrm{det} A = 1 \;\mbox{and} \; a_{33}>0\}.$$

The group $O_1(3,\r)$ has four connected components (see \cite{Alexandre, RL}), which are
$$O^{++}_1(3,\r) = \{A \in O_1(3,\r)\;|\; \textrm{det} A = 1 \;\mbox{and} \;  a_{33}>0\},$$
$$O^{+-}_1(3,\r) = \{A \in O_1(3,\r)\;|\;  \textrm{det} A = 1 \;\mbox{and} \;  a_{33}<0\},$$ 
$$O^{-+}_1(3,\r) = \{A \in O_1(3,\r)\;|\;  \textrm{det} A = -1 \;\mbox{and} \;  a_{33}>0\},$$
$$O^{--}_1(3,\r) = \{A \in O_1(3,\r)\;|\;  \textrm{det} A = -1 \;\mbox{and} \;  a_{33}<0\}.$$

We recall that it suffices to study $O^{++}_1(3,\r)$ (see \cite[Proposition 1.4.22]{Alexandre}) and an element $A$ of $O^{++}_1(3,\mathbb R)$ is similar to one of the three matrices (see \cite[Theorem 1.5.3]{Alexandre}):
\begin{enumerate}
  \item {\bf Hyperbolic rotation} (the direction of the rotation axis is a spacelike vector),
  \begin{equation*}A=\left[\begin{array}{ccc}
                     1 & 0 & 0 \\
                     0 & \cosh \theta & \sinh \theta \\
                     0 & \sinh \theta & \cosh \theta
                   \end{array}\right].
\end{equation*}
  \item  {\bf Elliptic rotation} (the direction of the rotation axis is a timelike vector),
  \begin{equation*}A=\left[\begin{array}{ccc}
   \cos \theta & -\sin \theta & 0 \\
   \sin \theta & \cos \theta & 0\\
      0       & 0          & 1 
                   \end{array}\right].
\end{equation*}
  \item  {\bf Parabolic rotation} (the direction of the rotation axis is a lightlike vector),
  \begin{equation*}A=\left[\begin{array}{ccc}
     1 & -\theta & \theta \\
   \theta & 1 -\frac{\theta^2}{2}& \frac{\theta^2}{2}\\
   \theta   & -\frac{\theta^2}{2}          & 1 + \frac{\theta^2}{2}
                   \end{array}\right].
\end{equation*}
\end{enumerate}
For the proof of the next result, we define the following functions:
\[c_k(\theta)=\left\{
\begin{aligned}
&\cos \theta, \quad \mbox{ if } k=-1,\\
&1,  \qquad \quad \mbox{  if } k=0,\\
&\cosh \theta, \quad \mbox{ if } k=1,
\end{aligned}
\right. \qquad \qquad
s_k(\theta)=\left\{
\begin{aligned}
&\sin \theta, \quad\; \mbox{ if } k=-1,\\
&-\varepsilon\,\theta, \quad \mbox{  if } k=0,\\
&\sinh \theta, \quad \mbox{ if } k=1.
\end{aligned}
\right .
\]
\begin{remark}\label{propcksk}The following properties hold:
\begin{itemize}
  \item[i)]$c_k^2(\theta)-k\:s_k^2(\theta)=1$,
  \vspace{0.25cm}
  \item[ii)]$c_k^2(\theta)+k\:s_k^2(\theta)=c_k(2\theta)$,
   \vspace{0.25cm}
  \item[iii)]$2s_k(\theta) c_k(\theta)=s_k(2\theta)$.
\end{itemize}
\end{remark}

 \begin{theorem}\label{teo2}
Let $a, b\in \mathbb{K}$ satisfying $a\bar a-\varepsilon\, b\bar b =1$, with $\bar b z+\bar a\notin \mathcal{C}$, and let  $T_{ab}:\mathbb K\rightarrow \mathbb K$ be the $\K$-bilinear transformation defined by  
$$T_{ab}(z)=\dfrac{a\,z+\varepsilon\, b}{\bar b\, z+\bar a}.$$
If $\pi$ is the stereographic projection given by \eqref{PEspacelike} (respectively, by \eqref{PEtimelike}) if $\varepsilon=1$ (respectively, $\varepsilon=-1$), then $\pi^{-1}\circ T_{ab}\circ \pi$ represents a rotation in $O_1^{++}(3,\r)$.
 \end{theorem}
\begin{proof}
 Let $P_0=(u_0,v_0,w_0)$ be a point of $\mathcal{H}^2_{\varepsilon}$ and $z_1=T_{ab}(z_0)$.
\subsection*{The case \texorpdfstring{$\varepsilon=1$}{e = 1}} The point $P_0$ corresponds, under the stereographic projection $\pi$ given by \eqref{PEspacelike},
to the complex number $z_0=\dfrac{u_0+ iv_0}{1-w_0}$,  with $w_0\neq 1$. Then
\begin{equation}\label{moebius}
   z_1=\frac{a\,(u_0+iv_0)+ b\,(1-w_0)}{\bar b\,(u_0+iv_0)+\bar a\, (1-w_0)},
   \end{equation}
   and there exists a point $P_1=(u_1,v_1,w_1)\in \mathcal{H}^2_{1}$ which corresponds under $\pi$ to the complex number $z_1$.
 In particular, we have  $z_1=\dfrac{u_1+iv_1}{1-w_1}$,  
 with
  \[u_1=\frac{-2\, \Re z_1}{\vert z_1\vert^2-1},\qquad  v_1=\frac{-2\, \Im z_1}{\vert z_1\vert^2-1},\qquad w_1=\frac{1+\vert z_1\vert^2}{\vert z_1\vert^2-1}.\]
  Therefore
 \[ \vert z_1\vert^2=- \frac{1+(\vert a \vert^2+\vert b \vert^2) w_0-2 \Re[a\bar b(u_0+iv_0)]}{1-(\vert a \vert^2+\vert b \vert^2) w_0+2 \Re[a\bar b(u_0+iv_0)]}
 \]
 and 
  \[ \vert z_1\vert^2-1= \frac{-2}{1-(\vert a \vert^2+\vert b \vert^2)\, w_0+2 \Re[a\bar b\,(u_0+iv_0)]}.
 \]
 Now, from \eqref{moebius}, we have
 \begin{equation}
 \begin{aligned}
z_1 =&\frac{(a^2+b^2)\,u_0+i(a^2-b^2)\,v_0-2ab\,w_0}{1-(\vert a \vert^2+\vert b \vert^2)\, w_0+2 \Re[a\bar b\,(u_0+iv_0)]}.
\end{aligned}
\end{equation}
 Then  
\begin{equation}\label{u_1,v_1}
u_1+iv_1=\frac{-2z_1}{\vert z_1 \vert^2-1}=(a^2+b^2)u_0+i\,(a^2-b^2)\,v_0-2ab\,w_0 \end{equation}
and 
\begin{equation}\label{w_1}
w_1=\frac{1+\vert z_1\vert^2}{\vert z_1\vert^2-1}=(\vert a \vert^2+\vert b \vert^2)\, w_0-2 \Re[a\bar b\,(u_0+iv_0)].\end{equation} 
To express the values of $u_1$, $v_1$ and $w_1$ in terms of $u_0$, $v_0$ and $w_0$ we need to calculate $ \vert a\vert^2+\vert b\vert^2$, $ a^2\pm b^2$,  $2ab$ and $2a\bar b$. 
We can choose the coefficients $a$ and $b$ like
  $$ a=c_k\left(\frac{\theta}{2}\right)-i\> r\>s_k\left(\frac{\theta}{2}\right),\qquad b=(q-i\>p)s_k\left(\frac{\theta}{2}\right),$$
  with $\theta, p, q,r \in\mathbb{R}.$
  As $\vert a\vert^2-\vert b\vert^2 =1$, then $p^2+ q^2 -r^2 = k$, where $k=1,0$ or $-1$.
  Now using Remark~\ref{propcksk} we obtain:
  \begin{eqnarray}
   \vert a\vert^2+\vert b\vert^2  &= &c_k(\theta)+2r^2 s_k^2\left(\frac{\theta}{2}\right),\nonumber \\
    a^2-b^2&=&c_k(\theta)-2(q^2-i\>q\>p)s_k^2\left(\frac{\theta}{2}\right)-i\>r s_k(\theta),\nonumber\\
    a^2+b^2&=&c_k(\theta)-2(p^2+i\>q\>p)s_k^2\left(\frac{\theta}{2}\right)-i\>r s_k(\theta),\nonumber \\
    2ab&=&-2(p\>r+i\>r\>q)s_k^2\left(\frac{\theta}{2}\right)+(q-i\>p) s_k(\theta),\nonumber\\
    2a\bar b&=&2(p\>r-i\>r\>q)s_k^2\left(\frac{\theta}{2}\right)+(q+i\>p) s_k(\theta).\nonumber  
    \end{eqnarray}
    
\subsection*{The case \texorpdfstring{$\varepsilon=-1$}{e = -1}}  The point $P_0$ corresponds under the stereographic projection $\pi$ given by \eqref{PEtimelike}
to the paracomplex number $z_0=\dfrac{-v_0+ \tau w_0}{1+ u_0},$ with $u_0\neq -1$. So
  \begin{equation}\label{moebiusL}
z_1=T_{ab}(z_0)=\frac{a\,(-v_0+ \tau w_0)-b\,(1+u_0)}{\bar b\,(-v_0+ \tau w_0)+\bar a \,(1+u_0)},
\end{equation}
 and there exists a point $P_1=(u_1,v_1,w_1)\in \mathcal{H}^2_{-1}$ which corresponds under $\pi$ to the paracomplex number $z_1$.
In particular, $$z_1=\dfrac{-v_1+\tau w_1}{1+u_1},$$
 with
\[u_1=\frac{1-z_1\bar z_1}{1+z_1\bar z_1},\qquad v_1=\frac{-2\Re z_1}{1+z_1\bar z_1},\qquad  w_1=\frac{2\Im z_1}{1+z_1\bar z_1}.\]
This gives 
 \[ z_1\bar z_1=\frac{1-(a\bar a-b \bar b) u_0-2 \Re[a\bar b(-v_0+\tau w_0)]}{1+(a\bar a-b \bar b) u_0+2 \Re[a\bar b(-v_0+\tau w_0)]}
 \]
 and 
  \[ 1+z_1\bar z_1= \frac{2}{1+(a\bar a-b \bar b) u_0+2 \Re[a\bar b(-v_0+\tau w_0)]}.
 \]
From \eqref{moebiusL}, it follows that
 \begin{equation}
 \begin{aligned}
z_1=\frac{-2abu_0+(b^2-a^2)v_0+\tau(a^2+b^2)w_0}{1+(a\bar a-b \bar b) u_0+2 \Re[a\bar b(-v_0+\tau w_0)]
}.
\end{aligned}
\end{equation}
Therefore
 \begin{equation}\label{u_1}
u_1=\frac{1- z_1\bar z_1}{1+z_1\bar z_1}=(a\bar a-b \bar b) u_0 + 2 \Re[a\bar b(-v_0+\tau w_0)]
\end{equation}
and 
\begin{equation}\label{v_1,w_1}
-v_1+\tau w_1=\frac{2z_1}{1+z_1\bar z_1}=-2abu_0+(b^2-a^2)v_0+\tau(a^2+b^2)w_0.
\end{equation} 
To express the values of $u_1$, $v_1$ and $w_1$ in terms of $u_0$, $v_0$ and $w_0$ we need to calculate $b^2\pm a^2$, $ a^2+b^2$,  $2ab$ and $2a\bar b$. 
We can choose the coefficients $a$ and $b$ as
  \[a=c_k\left(\frac{\theta}{2}\right)+\tau \> p\>s_k\left(\frac{\theta}{2}\right),\qquad b=(r-\tau \>q)s_k\left(\frac{\theta}{2}\right),\]
with $\theta, p, q, r\in\mathbb{R}$. Since $ a\bar a+ b \bar b =1$ then $p^2+ q^2 -r^2 = k$,  where $k=1,0$ or $-1$.
  Now from Remark~\ref{propcksk} we have:
\begin{eqnarray}
 a\bar a - b\bar b&=&c_k(\theta)-2p^2 s_k^2\left(\frac{\theta}{2}\right),\nonumber\\
 b^2-a^2&=&-c_k(\theta)+2(q^2-\tau\>q\>r)s_k^2\left(\frac{\theta}{2}\right)-\tau \>p\,s_k(\theta),\nonumber\\
 a^2+b^2&=&c_k(\theta)+2(r^2-\tau\>q\>r)s_k^2\left(\frac{\theta}{2}\right)+\tau\>p\, s_k(\theta),\nonumber\\ 2ab&=&2(-p\>q+\tau\>r\>p)s_k^2\left(\frac{\theta}{2}\right)+(r-\tau\>q) s_k(\theta),\nonumber\\ 
 2a\bar b&=&2(p\>q+\tau\>r\>p)s_k^2\left(\frac{\theta}{2}\right)+(r+\tau\>q) s_k(\theta).\nonumber 
\end{eqnarray}
\\ 
Finally, if $L=(p,q,r)$, using  \eqref{u_1,v_1}, \eqref{w_1}, \eqref{u_1} and \eqref{v_1,w_1}, it follows that
\begin{eqnarray}
  u_1 &=& c_k(\theta)\,u_0-2 p\langle P_0,L\rangle\, s_k^2\left(\frac{\theta}{2}\right) + \varepsilon s_k(\theta)\,(rv_0- q w_0),\nonumber\\
  v_1 &=& c_k\,(\theta)v_0-2q\langle P_0,L \rangle\, s_k^2\left(\frac{\theta}{2}\right) + \varepsilon s_k(\theta)\,(pw_0-ru_0),\nonumber\\
  w_1 &=& c_k\,(\theta)w_0-2r\langle P_0,L\rangle\, s_k^2\left(\frac{\theta}{2}\right)+ \varepsilon s_k(\theta)\,(pv_0-qu_0).\nonumber 
\end{eqnarray}
Consequently, we have found that
$$P_1=c_k(\theta)\,P_0-2s_k^2\left(\frac{\theta}{2}\right)\langle P_0,L\rangle L+\varepsilon\, s_k(\theta)\,(P_0\times L).$$

It is easy to check that $\langle P_1,P_1\rangle =-\varepsilon$ and so $\pi^{-1}\circ T_{ab}\circ \pi$ is pseudo-orthogonal transformation. Therefore, the point $P_1\in\mathcal{H}_\varepsilon^2$ is the image of the point $P_0$ under a pseudo-orthogonal transformation fixing the direction of $L$.
\end{proof}

\begin{remark}
If $L$ is a timelike vector ($k=-1$) then $\pi^{-1}\circ T_{ab}\circ \pi$ represents a hyperbolic rotation,
  if $L$ is a spacelike vector ($k=1$) then  $\pi^{-1}\circ T_{ab}\circ \pi$  is an elliptic rotation and if $L$ is a lightlike vector then $\pi^{-1}\circ T_{ab}\circ \pi$  is a parabolic rotation.
\end{remark}

\section{Minimal surfaces and Liouville's formula in \texorpdfstring{$\mathbb{L}^3$}{L3}} \label{sfive}

Let $\Omega\subset \K$ be an open set and $\psi:\Omega\rightarrow \L^3$ be a smooth, conformal spacelike (respectively, timelike) immersion with diagonalizable Weingarten map, that is,
$$E=\Vert\psi_x\Vert^2=\varepsilon\, \Vert \psi_y\Vert^2=\varepsilon\, G, \qquad F=\langle \psi_x,\psi_y \rangle=0$$ and  the induced metric is $$ds^2=E\,(dx^2+\varepsilon dy^2)=E\,\vert dz\vert^2.$$
The oriented normal vector 
$$N=\dfrac{\psi_x\times\psi_y}{E}$$ is such that $\langle N,N\rangle =-\varepsilon$.
Also, the second fundamental form is given by $l \:dx^2+2m \: dx dy+n \: dy^2$, where $$l=\langle \psi_{xx},N\rangle,\qquad m=\langle \psi_{xy},N\rangle,\qquad n=\langle \psi_{yy},N\rangle.$$
The mean curvature and the Gaussian curvature are given, respectively, by
\begin{equation}\label{HK}
H=-\dfrac{(\varepsilon\, l+n)}{2E}\qquad\mbox{and}\qquad K=\dfrac{m^2 - ln}{E^2}.
\end{equation}
Since the Weingarten map is diagonalizable (see \cite[Th. 3.4.6]{Alexandre}), we  have that  $H^2 + \varepsilon K\geq 0$ and 
\begin{equation}\label{HKprincipal}
H=-\varepsilon\,\dfrac{k_1 + k_2}{2},\qquad  K=-\varepsilon\, k_1k_2,
\end{equation}
where $k_1$ and $k_2$ are the principal curvatures. Given $z\in \Omega\subset\mathbb{K}$, then
\begin{equation}\label{fzfzs}
\langle \psi_z, \psi_z\rangle =0=\langle \psi_{\bar z},\psi_{\bar z}\rangle, \qquad 
\langle \psi_{z}, \psi_{\bar z}\rangle =\dfrac{E}{2}
\end{equation}
and
\begin{equation}\label{normal}
N=\dfrac{2 i \,\psi_{\bar z} \times \psi_z}{E}\quad\mbox{\Big(respectively}, N=\dfrac{2 \tau\,\psi_{\bar z} \times \psi_z}{E}\Big).
\end{equation}
Clearly, we have that
\begin{equation}
\langle N, \psi_z\rangle =0=\langle N,\psi_{\bar z}\rangle.
\end{equation}
\begin{proposition}\label{prop1}
In terms of a (para)complex coordinate $z\in\Omega$, the immersion $\psi:\Omega\rightarrow \L^3$ satisfies the following properties:
\begin{enumerate} 
\item[i)] $\Delta \psi=\dfrac{4 \psi_{z\bar z}}{E}=2H N$. In particular, $\psi$ is harmonic if and only if $\psi$ is minimal ($H=0$);
\item[ii)] $K=-\Delta(\log\sqrt{E})$ (Gauss equation);
\item[iii)] $2\,\langle \psi_{zz},N\rangle=\alpha$, where $$\alpha=\dfrac{l-n}{2}-i\, m\;\text{(respectively,}\; \alpha=\dfrac{l+n}{2}+\tau\, m);$$
\item[iv)] $\psi_{zz}=\dfrac{E_z}{E}\psi_z-\varepsilon\:\dfrac{\alpha}{2}N$;
\item[v)] $\langle \psi_{\bar z},N_z\rangle=EH/2$;
\item[vi)] $N_z=\varepsilon\:H \psi_z-\dfrac{\alpha}{E}\psi_{\bar z}$;
\item[vii)] $\alpha_{\bar z}=-\varepsilon\: E H_z$ (Codazzi equation);
\item[viii)] $\vert \alpha\vert^2=E^2(H^2+\varepsilon K)=\dfrac{E^2}{4}(k_1-k_2)^2$, where $k_1$,$k_2$ denote the principal curvatures.
\item[ix)] If $H$ is constant, then the Hopf differential $\mathcal H=\alpha(z) dz^2$ is a $\mathbb{K}$-holomorphic quadratic differential globally defined on $\Omega$. Moreover, the isolated zeros of $\mathcal H$ coincide with the umbilic points of the immersion.
\end{enumerate}
\begin{proof}
Observe that
\begin{equation}\label{laplace}
4\psi_{z\bar z}=\psi_{xx}+\varepsilon\psi_{yy}=E\Delta \psi.
\end{equation} 
Now using equation~\eqref{HK} it follows that
\begin{equation}\label{fzzbar}
\langle 4\psi_{z\bar z}, N\rangle=\langle\psi_{xx}+\varepsilon \psi_{yy}, N \rangle=l+\varepsilon n=-2\varepsilon EH.
\end{equation} 
On the other hand, from \eqref{fzfzs}  we have
 $$\langle\psi_{z\bar z},\psi_z\rangle=0\quad \mbox{and}\quad \langle\psi_{z\bar z},\psi_{\bar z}\rangle=0,$$ 
 so $\psi_{z\bar {z}}$ is orthogonal to $\psi_z$ and $\psi_{\bar z}$ and, thus, it is parallel to $N$. Therefore 
 $$4\psi_{z\bar z}=-\varepsilon \langle4\psi_{z\bar z},N\rangle N=2EH\,N$$ and using \eqref{laplace} we have proved $i)$.
For the proof of $ii)$, we use that
$$K=-\dfrac{1}{E}\left(\dfrac{\partial}{\partial x}\dfrac{\partial \log\sqrt{E}}{\partial x}+\varepsilon \dfrac{\partial}{\partial y}\dfrac{\partial \log \sqrt{E}}{\partial y}\right)= -\Delta \log(\sqrt{E}).$$
To show $(iii)$, we compute
\[ 4\psi_{zz} = \left\{
 \begin{aligned}
&(\psi_x-i\psi_y)_x-i(\psi_x-i\psi_y)_y=\psi_{xx}-\psi_{yy}-2i\psi_{xy},\quad  \mbox{ if } \varepsilon=1, \\ \\
&(\psi_x+\tau \psi_y)_x+\tau(\psi_x+\tau \psi_y)_y=\psi_{xx}+\psi_{yy}+2\tau \psi_{xy}, \quad \mbox{ if } \varepsilon=-1      
 \end{aligned} 
\right.
\]
and
\[ 2\alpha= \langle 4\psi_{zz}, N\rangle = \left\{
 \begin{aligned}
& l-n-2i\,m, \quad  \mbox{ if } \varepsilon=1, \\
& l+n +2\tau\, m, \quad \mbox{ if } \varepsilon=-1.      
 \end{aligned} 
\right.
\]
Thus we get
\[ \alpha= \left\{
 \begin{aligned}
& \dfrac{l-n}{2}-i\,m, \quad  \mbox{ if } \varepsilon=1, \\
& \dfrac{l+n}{2} +\tau\, m, \quad \mbox{ if } \varepsilon=-1.      
 \end{aligned} 
\right.
\]
For the assertion $iv)$, observe that from \eqref{fzfzs} follows that
$$\langle \psi_{zz}, \psi_z\rangle=0\quad \mbox{and} \quad \langle \psi_{zz}, \psi_{\bar z}\rangle=\dfrac{E_z}{2}.$$ 
Decomposing $\psi_{zz}=a\,\psi_z+b\,\psi_{\bar z}+c\,N$ and making the inner product with respect to $\psi_z$, $\psi_{\bar z}$ and $N$ we obtain
$$a=\dfrac{E_z}{E},\qquad b=0,\qquad c = -\varepsilon\:\dfrac{\alpha}{2},$$
that is,
 \begin{equation}
\psi_{zz}=\dfrac{E_z}{E}\psi_z-\varepsilon\dfrac{\alpha}{2} N.
\end{equation}
By differentiation of $\langle\psi_{\bar z},N\rangle=0$, we have  $\langle \psi_{z\bar z},N\rangle +\langle \psi_{\bar z},N_z\rangle=0$. So using \eqref{fzzbar} we find 
\begin{equation}\label{five}
\langle\psi_{\bar z},N_z\rangle =-\langle \psi_{z\bar z},N\rangle=\varepsilon\:\dfrac{EH}{2},
\end{equation} that is, $v)$.
To prove $vi)$, we consider the decomposition of $N_z=a\psi_z+b\psi_{\bar z}+cN$ and
we make the inner product with respect to $\psi_z$, $\psi_{\bar z}$ and $N$. Taking into account that $\langle \psi_{z},N_z\rangle=-\langle \psi_{zz}, N\rangle=-\alpha/2$, we obtain
$$a=\varepsilon\: H,\qquad b=-\dfrac{\alpha}{E},\qquad c = 0,$$
that is,
 \begin{equation}\label{six}
N_z=\varepsilon\: H\psi_z-\dfrac{\alpha}{E} \psi_{\bar z}.
\end{equation}
Now  differentiation of $\langle \psi_{zz}, N\rangle =\alpha/2$ gives
$$\dfrac{\alpha_{\bar z}}{2}=\langle \psi_{zz\bar z}, N\rangle +\langle \psi_{zz}, N_{\bar z}\rangle,$$
where
$$\langle\psi_{zz\bar z}, N\rangle=\langle\psi_{z\bar z}, N\rangle_z-\langle\psi_{z\bar z}, N_z\rangle=\Big(-\varepsilon\dfrac{EH}{2}\Big)_z-\dfrac{1}{2}\langle EHN, 
N_z\rangle=-\varepsilon\dfrac{(EH)_z}{2}.$$
Also, $v)$ yields
$$\langle \psi_{zz}, N_{\bar z}\rangle=\Big\langle\dfrac{E_z}{E}\psi_z-\varepsilon\dfrac{\alpha}{2} N,N_{\bar z}\Big\rangle=\dfrac{E_z}{E}\langle\psi_z,N_{\bar z}\rangle=\dfrac{E_z}{E}\overline{\langle\psi_{\bar z},N_{z}\rangle}=\varepsilon\:\dfrac{E_zH}{2}.$$
Putting together these expressions we have assertion $vii)$.
The proof of $viii)$ follows from equation~\eqref{HK}
$$\alpha\, \overline{\alpha}=\left(\frac{l-\varepsilon n}{2}\right)^2+\varepsilon\, m^2=\frac{(\varepsilon l +n)^2}{4}+\varepsilon(-ln+m^2)=E^2\,(H^2+\varepsilon K).$$
We have that the Weingarten map is diagonalizable, using equation~\eqref{HKprincipal}, so 
$$\vert \alpha\vert^2=E^2\,(H^2+\varepsilon\, K)=E^2\,\frac{(k_1-k_2)^2}{4}.$$

If $H$ is constant, one has that $H_z =0$, hence by $vii)$ $\alpha_{\bar z}=0$ and $\alpha$ is $\mathbb{K}$-holomorphic. Moreover, a change of (para)complex coordinates $z \rightarrow w $ converts $\alpha(z)=2\langle \psi_{zz},N\rangle$ into
$$2\,\langle (\psi(z))_{ww},N\rangle=2\,\Big\langle\Big(\dfrac{dz}{dw}\Big)^2\psi_{zz}+\Big(\dfrac{d^2z}{dw^2}\Big)\psi_{z}, N(z(w))\Big\rangle =\alpha(z(w))\Big(\dfrac{dz}{dw}\Big)^2.$$
Consequently, $\mathcal{H}=\alpha(z)dz^2$ is a global $\mathbb{K}$-holomorphic quadratic differential on $\Omega$. From iii) it follows that the isolated zeros of $\mathcal H$ coincide with the points where $l = \varepsilon\, n$ and $m=0$, that are exactly the points where the coordinate directions are principal (see \cite[Proposition 3.5.6]{Alexandre}) and, as $k_1=k_2$, the points are umbilical.
\end{proof}
\end{proposition}

Let $\psi:\Omega\subset \K \rightarrow \L^3$ be a smooth conformal spacelike (respectively, timelike) immersion with diagonalizable Weingarten map.
Now we need to consider the following quadratic differential form
$$\Omega_1=\varepsilon\, H \,{\rm I} + {\rm II},$$ where ${\rm I}$ and $\textrm{II}$ are the first and the second fundamental forms on $\psi(\Omega)=S$, respectively (see \cite{tilla, WolfEx} and \cite{Wolf}). 
The form $\Omega_1$ is a Lorentzian (respectively, Euclidean) metric on the open set $$S_{\Omega_1}=\{p\in S\;\vert\; k_1(p)\neq k_2(p)\}.$$

\begin{theorem}
Let $\psi:\Omega\subset \K \rightarrow \L^3$ be a smooth conformal spacelike (respectively, timelike) immersion with constant mean curvature $H$. We suppose that $H^2+\varepsilon K > 0$, where $K$ is the Gauss curvature and, also, we define on  $S_{\Omega_1}$ the function  
\begin{equation}\label{lambda-bis}
\lambda=-\frac{1}{4}\log(H^2+\varepsilon\, K).
\end{equation}
If  $p\in S_{\Omega_1}$, then there exists a neighborhood $U\subset S_{\Omega_1}$  of $ p$ such that
\begin{equation}
  {\rm I} = e^{2\lambda}\,(dx^2+\varepsilon\, dy^2), \label{I} \\  
\end{equation}
\begin{equation}\label{II}
  {\rm II} = (1-\varepsilon\, H e^{2\lambda})\,dx^2-\varepsilon (1+\varepsilon\, H e^{2\lambda})\,dy^2,\\
\end{equation}
\begin{equation}
k_1 = e^{-2\lambda}-\varepsilon\, H,\qquad k_2=-(e^{-2\lambda}+\varepsilon\, H),  \qquad \varepsilon K = e^{-4\lambda} - H^2.
\end{equation}
\end{theorem}

\begin{proof}
Let $H'=\sqrt{H^2+\varepsilon\,K}$ be the skew curvature. From \eqref{lambda-bis}, it follows that $H'=e^{-2\lambda}$ and using \cite[Lemma 13]{tilla} we have that there exists a neighborhood $U$ of $p$, with $U\subset S_{\Omega_1}$, such that
\[
H'\,\textrm{I}=dx^2 +\varepsilon\, dy^2 \qquad \mbox{and} \qquad H'\, \textrm{II}=(H' -\varepsilon H)\,dx^2 -\varepsilon\,(H' +\varepsilon H)\,dy^2.
\]
Therefore, the fundamental forms are given by \eqref{I} and \eqref{II}, and  $E=\varepsilon\, G = e^{2\lambda}$, $l=k_1\,E$, $n=k_2\,G$. Also, the following holds true:
$$\begin{aligned}
  l &=1-\varepsilon\, H e^{2\lambda}, \\
  n & =-\varepsilon\,(1+\varepsilon\, H e^{2\lambda}),\\
  k_1 &= e^{-2\lambda}-\varepsilon\, H, \\
  k_2& =-e^{-2\lambda}-\varepsilon\, H,\\
  K &=\varepsilon\, (e^{-4\lambda}-H^2).
\end{aligned}
$$
\end{proof}

\begin{corollary}\label{liouvillecoordinates}
Let $\psi:\Omega\subset \K \rightarrow \L^3$ be a smooth minimal conformal spacelike (respectively, timelike) immersion. We suppose that $\varepsilon K > 0$, where $K$ is the Gauss curvature and we consider the function  
$$\lambda=-\frac{1}{4}\log(\varepsilon\, K),$$ defined on $S_{\Omega_1}$. If  $p\in S_{\Omega_1}$, then there exists a neighborhood $U\subset S_{\Omega_1}$ of $p$ such that
\begin{eqnarray*}
 {\rm I} &=& e^{2\lambda}\,(dx^2+\varepsilon\, dy^2), \\
 {\rm II} &=& dx^2-\varepsilon\, dy^2,\\
 k_1 &=& e^{-2\lambda},\qquad k_2=-e^{-2\lambda},\qquad K = \varepsilon\, e^{-4\lambda}.
\end{eqnarray*}
\end{corollary}
\begin{definition}
Given a minimal spacelike (respectively, timelike) surface with Gauss curvature such that $\varepsilon K > 0$, we call a {\it Liouville parameter} a conformal parameter $z$ for which the first and second fundamental forms are given by Corollary~\ref{liouvillecoordinates}.
\end{definition}

\begin{proposition}\label{prop4}
Let $\psi:\Omega\subset \K \rightarrow \L^3$ be a smooth minimal conformal spacelike (respectively, timelike) immersion with Gauss curvature such that $\varepsilon\, K > 0$. Suppose that $z\in \Omega $ is a Liouville parameter. Then the coordinate lines are lines of curvature  and the conformal factor $e^{2\lambda}$ corresponds to a solution $\lambda$ of the Liouville equation
\begin{equation}\label{liouvilleequation}
\Delta\lambda = -\varepsilon\, e^{-4\lambda}.
\end{equation}
\end{proposition}
\begin{proof}
Using Corollary~\ref{liouvillecoordinates} and item $ii)$  of Proposition~\ref{prop1} we have that the coordinate lines are lines of curvature and
$$ \varepsilon\, e^{-4\lambda} = K = -\Delta\log(e^\lambda).$$
Thus we have a solution $\lambda$ of the Liouville equation.
\end{proof}

\begin{proposition}\label{prop3}
In terms of a Weierstrass pair, the geometric invariants of the minimal spacelike (respectively, timelike) immersion are expressed as follows:
\begin{itemize}
  \item[i)] The conformal factor is $$e^{2\lambda} = \dfrac{f\overline{f}\,(1 -\varepsilon\, g\overline{g})^2}{4}, $$ where $f\overline{f}>0$ and $g\overline{g}\neq \varepsilon$.
  \item[ii)] The oriented normal $N$ satisfies $\pi(N)=g$, where $\pi$ denotes the stereographic projection given by \eqref{PEspacelike} (respectively, \eqref{PEtimelike}).
  \item[iii)] The Hopf differential is $\mathcal{H}=-\varepsilon fg'dz^2$. Moreover, if $g$ has a pole of order $m\geq 1$ at $z_0$, then either $\mathcal{H}$ does not vanish at $z_0$, when $m=1$, or it  has a zero of order $m>1$ at $z_0$.  
\end{itemize} 
\end{proposition}
\begin{proof}
The proof of $i)$ can be seen in \cite{Kob} and \cite{konderak}. For assertion $ii)$, using \eqref{normal} we have that
\begin{equation}\label{NGauss}
 N=\left\{\begin{array}{cc}
  \Big( \dfrac{-2\,\Re(g)}{|g|^2 - 1}, \dfrac{-2\,\Im(g)}{|g|^2 - 1},\dfrac{1 + |g|^2}{|g|^2 - 1}\Big),\quad if \quad \varepsilon =1, \\ \\
\Big( \dfrac{1-g\overline{g}}{1 + g\overline{g}}, \dfrac{-2\,\Re(g)}{1 + g\overline{g}},\dfrac{2\,\Im(g)}{1 + g\overline{g}}\Big),\quad if \quad \varepsilon =-1.
 \end{array}\right.   
\end{equation}
Therefore,  $\pi(N)=g$, where $\pi$ denotes the stereographic projection given by \eqref{PEspacelike} (respectively, by \eqref{PEtimelike}).
\end{proof}
From the \eqref{RW} we have that
\begin{equation}
 4\psi_{zz} =\left\{
 \begin{array}{cc}
 f'\,(1 +g^2,i(1-g^2),-2g) + 2fg'(g,-gi,-1),\quad if \quad \varepsilon=1,\vspace{0.5cm}\\
 f'\,(2g, 1  -g^2,\tau (1+g^2)) + 2fg'(1, -g,g\tau), \quad if \quad \varepsilon=-1.
  \end{array}
 \right.
\end{equation}
By $iii)$ of Proposition~\ref{prop1} we have that
\begin{equation}\label{alpha}
\alpha= 2\langle \psi_{zz},N\rangle = -\varepsilon fg'.
\end{equation}
Then $\mathcal{H}=-\varepsilon fg' dz^2$.
If $g$ has a pole of order $m\geq 1$ at $z_0$, then $f$ has a zero of order $2m$ at that pole, hence $\alpha =-\varepsilon fg'$ admits an expansion around $z_0$ with exponent $m-1\geq 0$. That concludes the proof.
\begin{proposition}\label{prop2}
Any simply connected spacelike (respectively, timelike) minimal surface without umbilic points in $\L^3$ with Gauss curvature $\varepsilon K>0$ admits a Liouville coordinates $z\in\Omega\subset \mathbb{K}$ such that:
\begin{itemize}
  \item[i)]the coordinate lines are lines of curvature; 
  \item[ii)]the stereographic projection of the oriented normal $N$ is a meromorphic function $g$ defined in $\Omega$ admitting only simple poles and satisfying $g\overline{g}\neq \varepsilon$ and $g'\overline{g'}>0$ at all regular points;
  \item[iii)] the conformal factor $e^{2\lambda}$ determines a solution $\lambda$ of the Liouville equation, namely 
  $$e^{2\lambda} = \dfrac{(1 -\varepsilon\, g\overline{g})^2}{4g'\overline{g'}}.$$
\end{itemize}
\end{proposition}
\begin{proof}
 From Corollary~\ref{liouvillecoordinates} it follows that there are Liouville coordinates $z\in\Omega$, thus the Hopf differential is given by $\mathcal{H}=dz^2$, in other words $\alpha =1$. We consider the minimal immersion in Liouville coordinates, say $\psi:\Omega\rightarrow \mathbb{\L}^3$, and we use \eqref{fg} to define $f$ and $g$. Using \eqref{alpha} one has that 
\begin{equation}\label{f}
f = - \dfrac{\varepsilon}{g'}.
\end{equation}
By substituting \eqref{f} in the expression of $e^\lambda$ in item $i)$ of  Proposition~\ref{prop3}, we arrive at
\begin{equation}\label{solution}
e^{2\lambda} = \dfrac{(1 -\varepsilon\, g\overline{g})^2}{4 g'\overline{g'}}. 
\end{equation}
Applications of Propositions~\ref{prop4} and \ref{prop3} prove $i)$, $ii)$ and $iii)$.
\end{proof}

Conversely,  we have the following result.
\begin{proposition}\label{prop5}
A solution $\lambda (x,y)$ of the Liouville equation 
$$\Delta \lambda =-\varepsilon\, e^{-4\lambda}$$
defined in a simply connected region $\Omega$ determines, up to a rigid motion of $\mathbb{L}^3$, a nonumbilic minimal spacelike (respectively, timelike) immersion $\psi:\Omega\rightarrow \mathbb{L}^3$ with
a Liouville complex parameter (respectively, paracomplex) $z$ and conformal factor $e^\lambda$ as in \eqref{solution} such that:
\begin{itemize}
  \item[i)]the coordinate lines are lines of curvature; moreover, the principal curvatures are $k_1=-k_2 = e^{-2\lambda}$ and the Gaussian curvature is $K=-\varepsilon\, k_1k_2= \varepsilon\, e^{-4\lambda}$ and $\varepsilon\, K>0$, so $dN_p$ is diagonalizable;
  \item[ii)]the Weierstrass pair determined by $\psi$ is given by $$\Big(-\dfrac{\varepsilon}{g'}\,dz, g\Big),$$ where $g$ is a (para)meromorphic function  which admits only simple poles and such that $g'\overline{g'}>0$ and $g\overline{g}\neq 1$;
  \item[iii)]the solution $\lambda(x,y)$ and the (para)meromorphic function $g$ are related in $\Omega$ by \eqref{solution}.
\end{itemize}
\end{proposition}
\begin{proof}
Let $\lambda(x,y)$ be a solution of \eqref{liouvilleequation} defined in $\Omega$. We define 
$${\rm I}= e^{2\lambda}\,(dx^2 + \varepsilon\, dy^2)\qquad \mbox{and}\qquad {\rm II}= dx^2 - \varepsilon\, dy^2,$$ as candidates to first and second fundamental forms in conformal parameters $(x,y)$. The Christoffel symbols are easily computed (see \cite{Alexandre}) and are given by
\begin{equation}\label{christoffel}
\Gamma_{11}^2 = \Gamma_{12}^2 = -\varepsilon \Gamma_{22}^1 = \lambda_x, \qquad
\Gamma_{12}^1 = \Gamma_{22}^2 = -\varepsilon \Gamma_{11}^2 = \lambda_y.
\end{equation}
The Gauss equation is reduced to
$$-e^{2\lambda}\,K = \lambda_{xx} + \varepsilon\,\lambda_{yy},$$
which is satisfied, since $\lambda$ is a solution of \eqref{liouvilleequation} and $k_2=-k_1=e^{-2\lambda}$, hence $K=\varepsilon\, e^{-4\lambda}$. So $\varepsilon\, K = e^{-4\lambda} > 0.$ Also, the equations of Mainard-Codazzi are reduced to 
$$\Gamma_{12}^1 + \varepsilon\, \Gamma_{11}^2 = 0 \qquad \text{and}\qquad \Gamma_{22}^1 + \varepsilon\, \Gamma_{12}^2=0,$$
which are automatically satisfied, from the relations \eqref{christoffel} and the fact that the coefficients of ${\rm II}$ are $l=1$, $n=- \varepsilon$ and $m=0$.

Since the $\det {\rm I} = \varepsilon\, e^{2\lambda}\neq 0$ then the fundamental theorem on surfaces (\cite{Alexandre, tila}) guarantees that each convex subset $\Omega_0$ of $\Omega$ admits a differentiable immersion $\psi: \Omega\rightarrow \mathbb{L}^3$ such that the induced metric has arclength element $ds = e^{\lambda}\,|dz|$ and the second fundamental form is ${\rm II}$. It follows from our construction that the immersion is minimal, for $n + \varepsilon\, l= 0$; and without umbilic points, for $\varepsilon\, K >0$. Besides, those are Liouville coordinates and by Proposition~\ref{prop2} one has \eqref{solution}. Now we use the analytic continuation principle to extend the $\L$-holomorphic components of $\psi_z$, as well as the (para)meromorphic $g$, to the whole region $\Omega$. To complete the proof, one applies Propositions~\ref{prop3} and \ref{prop2}.
\end{proof}

\begin{theorem}\label{teo3}
		A solution of the Liouville equation is given by 
		\begin{equation}\label{liouville2}
			e^{\lambda(x,y)} = \dfrac{|1 - \varepsilon\, g(z)\overline{g}(z)|}{2\sqrt{g'(z)\,\overline{g'}(z)}}, \qquad z\in\Omega\subset\K,
		\end{equation}  
		where \( g \) is a (para)meromorphic function such that \( g(z)\overline{g}(z) \neq \varepsilon \), \( g'(z)\overline{g'}(z) > 0 \), and it has only simple poles. Moreover, \( g \) and its transformations
		\[
		T_{ab}(g) = \dfrac{a\,g + \varepsilon\, b}{\bar b\, g + \bar a},
		\]
		with \( a\bar a - \varepsilon\, b\bar b = 1 \) and \( \bar b\, g + \bar a \notin \mathcal{C} \), yield all solutions of \eqref{liouville2}.
		Conversely, let \( \Omega \subset \K \) be a simply connected open set.  Given a (para)meromorphic function \( g:\Omega \to \K \) such that \( g(z)\overline{g}(z) \neq \varepsilon \) and \( g'(z)\overline{g'}(z) > 0 \), with \( \varepsilon = \pm 1 \), then the function \( \lambda:\Omega \to \mathbb{R} \) defined by \eqref{liouville2} is a solution of the Liouville equation \( \Delta \lambda = -\varepsilon e^{-4\lambda} \), where \( \Delta \) denotes the Beltrami–Laplace operator with respect to the metric \( e^{2\lambda}(dx^2 + \varepsilon\, dy^2) \).
	\end{theorem}
    
\begin{proof}
	By Proposition~\ref{prop5}, a solution $\lambda$ determines a unique minimal immersion $\psi$, up to a rigid motion of $\L^3$. Let us modify $\psi$ accordingly, so that 
	$$\psi_2 = R\circ\psi + P_0,$$
	where $R$ is a linear pseudo-orthogonal transformation of $\L^3$ and $P_0$ is a constant vector. If $z\in\K$ is a Liouville parameter for $\psi$, then the Hopf differential of $\psi_2$ satisfies
	$$\alpha_2 = 2\langle R\circ \psi_{zz} ,R\circ N \rangle = \langle \psi_{zz}, N \rangle = \alpha =1$$
	yielding that $z$ is a Liouville parameter for $\psi_2$. The (para)meromorphic function $g=\pi(N)$ is changed into $$\pi\circ R\circ\pi^{-1}= T_{ab}(g),$$
	where in virtue of Theorem~\ref{teo2} the pseudo-transformation is given by
	$$T_{ab}(g) = \dfrac{a\,g+\varepsilon\, b}{\bar b\, g+\bar a},\qquad \mbox{with}\quad a\bar a-\varepsilon\, b\bar b =1, \quad \bar b g+\bar a\notin \mathcal{C}.$$
	
	Conversely, let $\Omega \subset \K$ be a simply connected open set, and let $z = x + i\,y$ (respectively, $z = x + \tau\,y$) be a complex (respectively, paracomplex) parameter in $\Omega$. For any function $f:\Omega \to \K$, we denote $f_z = \frac{\partial f}{\partial z}$ and $f_{\overline{z}} = \frac{\partial f}{\partial \overline{z}}$. When $f$ is (para)meromorphic, we have $f_{\overline{z}} = 0$, and thus we write $f_z = f'$.
	
	Consider a (para)meromorphic function $g:\Omega \to \K$ such that $g(z)\overline{g}(z) \neq \varepsilon$ and $g'(z)\overline{g'}(z) > 0$, with $\varepsilon = \pm 1$. Under these hypotheses, we define the real-valued function $\lambda:\Omega \to \mathbb{R}$ by
	\[
	\lambda(z) = \frac{1}{2} \log \left( \frac{(1 - \varepsilon g(z)\overline{g}(z))^2}{4 g'(z)\overline{g'}(z)} \right).
	\]
	
	We claim that the function $\lambda$ defined above satisfies the following differential identity:
	\[
	2 (e^{2\lambda})_{z\overline{z}} - 4 \lambda_{\overline{z}} (e^{2\lambda})_{z} = -\varepsilon.
	\]
	
	To verify this, we begin by computing the partial derivatives involved. Since $g$ is (para)meromorphic, the following formulas hold:
	\begin{align*}
		4 \lambda_{\overline{z}} &= - \frac{2}{(1 - \varepsilon g \overline{g}) \overline{g'}} \left( 2 \varepsilon g \overline{g'}^2 + (1 - \varepsilon g \overline{g}) \overline{g''} \right), \\
		(e^{2\lambda})_z &= - \frac{1 - \varepsilon g \overline{g}}{4 g'^2 \overline{g'}} \left( 2 \varepsilon \overline{g} g'^2 + (1 - \varepsilon g \overline{g}) g'' \right), \\
		2 (e^{2\lambda})_{z\overline{z}} &= - \frac{1}{2 g'^2 \overline{g'}^2} \Big\{ \overline{g'} \left[ -4 g \overline{g} g'^2 \overline{g'} + 2 \varepsilon g'^2 \overline{g'} - 2 \varepsilon (1 - \varepsilon g \overline{g}) g \overline{g'} g'' \right] \\
		&\quad - 2 \varepsilon \overline{g} (1 - \varepsilon g \overline{g}) g'^2 \overline{g''} - (1 - \varepsilon g \overline{g})^2 g'' \overline{g''} \Big\}.
	\end{align*}
	
	By a straightforward computation using the expressions above, the claim follows.
	
	Let us now consider the conformal metric on $\Omega\subset \K$ defined by $ds^{2} = e^{2\lambda}\,(dx^{2} + \varepsilon\, dy^{2})$. 
	With respect to this conformal metric, the Beltrami–Laplace operator takes the form given in \eqref{beltrami}.
	We can now compute $\Delta \lambda$ using the chain rule. By differentiating the logarithmic expression of $\lambda$, we obtain:
	\[
	4 \lambda_{z\overline{z}} = \frac{2 (e^{2\lambda})_{z\overline{z}} - 4 \lambda_{\overline{z}} (e^{2\lambda})_{z}}{e^{2\lambda}}.
	\]
	Substituting the identity we previously verified, we conclude that
	\[
	4 e^{-2\lambda} \lambda_{z\overline{z}} = -\varepsilon e^{-4\lambda}.
	\]
	This completes the proof that $\lambda$ is a solution to the Liouville equation associated with the conformal metric $ds^2$.
\end{proof}

\section{Construction of minimal surfaces}\label{ssix}
This section is devoted to the construction of minimal immersions in $\L^3$ using Theorem~\ref{teo3}. Specifically, we start from a solution of the Liouville equation~\eqref{liouville2} and we make explicit $g$, omitting long straightforward computations. We point out that the parametrizations are given in the Liouville parameters, so that $f(z)= -\varepsilon/g'(z)$ in the immersion formula~\eqref{RW}.  
\subsection{Spacelike minimal surfaces}
 \begin{example}[Spacelike Enneper's surface]
     In  $\Omega=\{x+iy\in\C\,|\, x^2 + y^2 \neq 1\} $ we consider the solution $$\lambda(x,y)=\log\Big(\dfrac{|1 - x^2 - y^2|}{2}\Big),$$ represented by $g(z) =z$, with $z \in \Omega$. We note that $g'(z)=-f(z)=1$ and $|g|^2 \neq 1$.
Then, using the Weierstrass representation~\eqref{RW}, we obtain the minimal immersion given by
$$\psi(x,y)=\Big(-\dfrac{1}{2}\Big(x + \dfrac{x^3}{3} - y^2x\Big),\dfrac{1}{2}\Big(y + \dfrac{y^3}{3} - x^2y\Big), \dfrac{x^2 - y^2}{2}\Big),$$
which is the {\em spacelike Enneper's surface of the first kind} (see \cite{Kob}).

\end{example}

\begin{example}[One-variable solution and spacelike catenoid of 1st kind]\label{1SC}
Let us consider the open set $\Omega~=~\{x+iy\in \C :x > 0\}$ and the solution $\lambda(x,y) =\log(\sinh x)$, which depends on one variable and it is represented by $g(z)=-e^z$, with $z \in \Omega$. We observe that $g'(z)\neq 0$, $|g|^2 = e^{2x}\neq 1$ and $f(z)=e^{-z}$.
The Weierstrass representation formula~\eqref{RW} gives the minimal immersion 
$$\psi(x,y)= (\sinh x\cos y, \sinh x\sin y,x),$$
which represents the {\em catenoid of 1st kind}, also called {\em elliptic catenoid}, described in \cite{CintraOnnis} as
$$x_1^{2}+x_2^{2}=\sinh^{2} x_3, \quad x_3> 0.$$
\end{example}

\begin{example}[Minkowski-Bonnet minimal surfaces]
First of all,  by \cite{FManhart} a {\it Minkowski-Bonnet spacelike surface} without umbilic points in $\L^3$ is a minimal spacelike surface whose lines of curvature are plane curves. In this example we take $$\Omega=\Big\{x+iy\in\C\,|\,\sinh x >\dfrac{b}{a}\Big\}$$ and the two-parameter family of solutions given by 
$$\lambda(x,y)=\log(a\sinh x + b\cos y),$$
with $a^2 + b^2=1$, $0<a\leq 1$, $0\leq b< 1$, 
which is periodic in one variable and represented by 
$g(z)=-ae^z-b, z \in \Omega.$
Note that $$g'(z)\neq 0,\qquad 1 - |g|^2 = a^2\,(1- e^{2x}) -2ab\,e^x\cos y\neq 0.$$  Moreover, \eqref{f} gives $f(z)=e^{-z}/a$.
Using the Weierstrass representation~\eqref{RW}, the corresponding immersion is given by
\begin{equation}\label{MBS}
\psi(x,y)= \Big(\big(-\frac{e^{-x}}{a}+a\cosh x \big)\cos y , a\sinh x \sin y,x\Big)+b\Big(x,y,-\frac{e^{-x}}{a}\cos y \Big).
\end{equation}
We note that $a=1$, $b=0$ yields the {\em spacelike catenoid of 1st kind} of Example~\ref{1SC}.
By Proposition~\ref{prop5} the lines of curvature are the coordinate lines. We have to show that $\alpha(x)=\psi(x,y_0)$ and  $\beta(y)=\psi(x_0,y)$ are plane curves. A straightforward calculation proves that $\alpha',\alpha'',\beta'$ and $\beta''$ are spacelike vectors. Moreover $\alpha',\alpha''$ are linearly independent and the space generated is a spacelike subspace of $\L^3$. Then $\alpha$ is an admissible curve with torsion (see \cite{Alexandre}) given by
$$\tau_\alpha =-\dfrac{\langle\alpha'\times\alpha''',\alpha''\rangle}{||\alpha'\times\alpha''||^2} = 0.$$
Analogously, $\beta$ is an admissible curve and its torsion is $\tau_\beta=0$.
Then the lines of curvature are plane curves and the minimal immersion~\eqref{MBS} gives a Minkowski-Bonnet surface.
\begin{figure}[ht]\label{fig-MB}
\begin{center}
\subfigure{\includegraphics[width=0.28\textwidth]{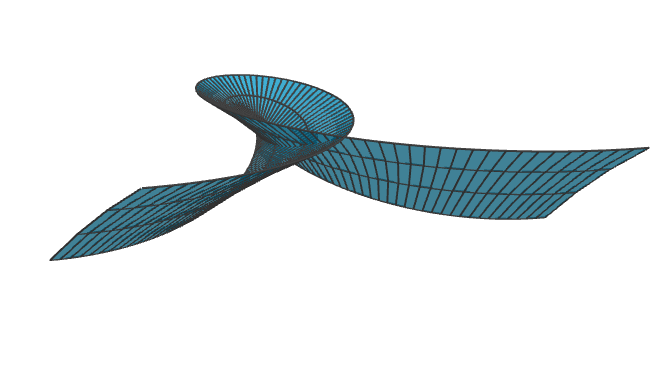}}
\hspace{0.5cm}
\subfigure{\includegraphics[width=0.23\textwidth]{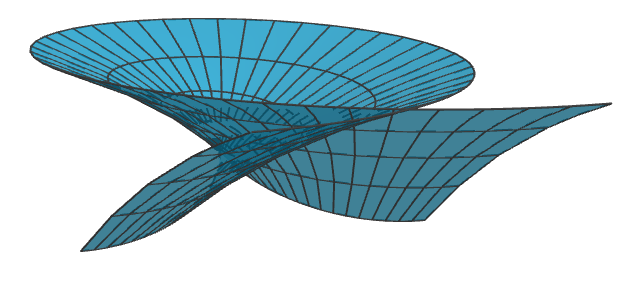}}
\hspace{0.5cm}
\subfigure{\includegraphics[width=0.24\textwidth]{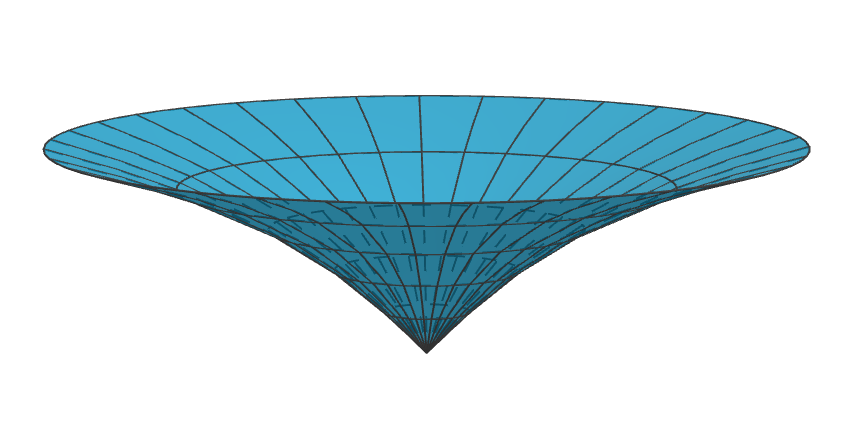}}
\end{center}\caption{Minkowski-Bonnet spacelike surfaces obtained for $a=0.5$, $a=0.8$ and $a=1$ (elliptic catenoid), respectively.}
\end{figure}
\end{example}

\begin{remark}
The spacelike minimal immersions $\psi,\psi^*:\Omega \subset \C\rightarrow \mathbb{R}^3$ are said to be {\it conjugate} (see \cite{Alexandre}) if and only if
$$\left\{\begin{aligned}
\psi_u &= \psi^*_v,\\
\psi_v &= -\psi^*_u. 
\end{aligned}\right.$$  
It is easy to check that if $z$ is a Liouville parameter for $\psi$, then 
$$z^*=\exp\Big(-\dfrac{\pi}{4}i\Big)z$$
is a Liouville parameter of the conjugate immersion $\psi^*$ and 
$$g^*(z^*)=g(\sqrt{i}z^*),\qquad \sqrt{i}=\exp\Big(\dfrac{\pi}{4}i\Big).$$ 
Moreover, the solution of the Liouville equation for $\psi^*$ is
$$\lambda^*(x^*,y^*)= \lambda\Big(\dfrac{x^*-y^*}{\sqrt{2}},\dfrac{x^*+y^*}{\sqrt{2}}\Big).$$ 
\end{remark} 

\begin{example}[Spacelike helicoid of 1st kind]\label{HS1}
Now, we describe the conjugate surface of the elliptic catenoid given in Example~\ref{1SC}, whose image in $\r^{3}$ is an open subset of the classical minimal helicoid $x_1\,\cos x_3+x_2\,\sin x_3=0$ (see \cite{CintraOnnis}).
For this, we consider $\Omega=\{x+iy\in \C \,:\, x > y\}$ and the solution $$\lambda(x,y)~ =~\log\Big(\sinh\Big(\dfrac{x-y}{\sqrt{2}}\Big)\Big)$$ represented by $g(z)=-\exp{(\sqrt{i}z)}$, $z \in \Omega$. We have that $$g'(z)\neq 0, \qquad |g|^2 = \exp(\sqrt{2}(x-y))\neq 1.$$  Moreover, \eqref{f} gives $f(z)=\exp(-\sqrt{i}z)/\sqrt{i}$.
Using the Weierstrass representation~\eqref{RW}, the corresponding spacelike immersion is given by
$$\psi(x,y)= \Big(\sin\Big(\frac{x+y}{\sqrt{2}}\Big)\cosh\Big(\frac{x-y}{\sqrt{2}}\Big), -\cos\Big(\frac{x+y}{\sqrt{2}}\Big)\cosh\Big(\dfrac{x-y}{\sqrt{2}}\Big),\frac{x+y}{\sqrt{2}}\Big)$$
and it represents the {\em helicoid of 1st kind} parameterized by
$$\psi(\sqrt{-i}\,z)=(\sin y\cosh x,-\cos y\cosh x,y).$$
\end{example}

\begin{example}[Minkowski-Thomsen minimal surfaces] Following \cite{FManhart},  {\it Minkowski-\linebreak Thomsen minimal} spacelike surfaces in $\L^3$ (without umbilic points) are the conjugates of Minkowski-Bonnet minimal surfaces. We consider $0<a\leq 1$, $0\leq b < 1$, with $a^2 + b^2=1$, and the two-parameter family of solutions given by
$$\lambda(x,y)=\log\bigg(a\sinh\Big(\dfrac{x-y}{\sqrt{2}}\Big) + b\cos\Big(\dfrac{x+y}{\sqrt{2}}\Big)\bigg),$$
defined in $$\Omega=\Big\{x+iy\in\C\,|\,\sinh\Big(\dfrac{x-y}{\sqrt{2}}\Big)>\dfrac{b}{a}\Big\}$$  
and represented by 
$g(z)=-(a\exp(\sqrt{i}z)+b),$ $z \in \Omega.$
Thus, $$g'(z)\neq 0,\qquad f(z)=\dfrac{\exp{(-\sqrt{i}z})}{a\sqrt{i}}.$$
Also, we obtain that $$1 - |g|^2 = a^2\,[1- \exp(\sqrt{2}(x-y))] -2ab\exp\Big(\frac{x-y}{\sqrt{2}}\Big)\cos\Big(\frac{x+y}{\sqrt{2}}\Big)\neq 0.$$
Using the Weierstrass representation~\eqref{RW}, the corresponding immersion is given by
\begin{eqnarray*}
\psi(x,y)=\Bigg(\bigg(\dfrac{1}{a}\exp\Big(\dfrac{-x+y}{\sqrt{2}}\Big)+a\sinh\Big(\dfrac{x-y}{\sqrt{2}}\Big)\bigg)\sin\Big(\dfrac{x+y}{\sqrt{2}}\Big),\\ -a\cosh\Big(\dfrac{x-y}{\sqrt{2}}\Big)\cos\Big(\dfrac{x+y}{\sqrt{2}}\Big),\dfrac{x+y}{\sqrt{2}}\Bigg)\\+b\,\Bigg(\dfrac{x+y}{\sqrt{2}},
\dfrac{-x+y}{\sqrt{2}},\dfrac{1}{a}\exp\Big(\dfrac{-x+y}{\sqrt{2}}\Big)\sin\Big(\dfrac{x+y}{\sqrt{2}}\Big)\Bigg).    
\end{eqnarray*}
Hence
$$\psi(\sqrt{-i}\,z)= \bigg(\Big(\frac{e^{-x}}{a} + a\sinh x\Big)\sin y, -a\cosh x\cos y,y\bigg) + b\bigg(y,-x,\frac{e^{-x}}{a}\sin y\bigg).$$
We note that $a=1$ and $b=0$ yields the helicoid of 1st kind (see Example~\ref{HS1}).
\end{example}

\subsection{Timelike minimal surfaces}

\begin{example}[Timelike Enneper's surface]
In this example, we take  the solution $$\lambda(x,y)=\log\Big(\dfrac{|1 + x^2 - y^2|}{2}\Big),$$ defined in $\Omega=\{x+\tau y\in\L\,|\, x^2 - y^2 \neq -1\} $, that is,  represented by $g(z) =z$, $z \in \Omega$. Then $g'(z)=f(z)=1$ and $g\overline{g} \neq -1$.
Using the Weierstrass representation~\eqref{RW}, the corresponding immersion is given by
$$\psi(x,y)=\bigg(\dfrac{x^2+y^2}{2}, \dfrac{1}{2}\Big(x - \dfrac{x^3}{3} - y^2x\Big),\dfrac{1}{2}\Big(y + \dfrac{y^3}{3} + x^2y\Big)\bigg),$$
which is the {\em timelike Enneper's surface} (see \cite{konderak}).
\end{example}

\begin{example}[One-variable solution and timelike catenoid of 1st kind]
Set $\Omega=\L$. The solution $\lambda(x,y) =\log(\cosh x)$ depends on one variable and is globally represented by $g(z)=-\mathsf{exp}(z),$ 
with $z \in \Omega$. So $g'(z)=-\mathsf{exp}(z)\neq 0$. We obtain $g\overline{g} = e^{2x}\neq -1$.  We have, by \eqref{f}, $f(z)=-\mathsf{exp}(-z)$.
Using the Weierstrass representation~\eqref{RW}, the corresponding immersion is given by
$$\psi(x,y)= (x,\cosh x \cosh y, -\cosh x\sinh y)$$
and it represents the {\em hyperbolic catenoid of 1st kind} described in \cite{CintraOnnis}:
$$x_2^{2} -x_3^{2}=\cosh^2 x_1.$$
\end{example}

\begin{example}[Timelike minimal Bonnet-type surface]
Let $a,b\in\mathbb{R}$  such that $a\geq 1$, $b \leq 0$ and $a^2 - b^2=1$. In $$\Omega=\Big\{x+\tau y\in\L\,|\,a\cosh x> -b\cosh(y)\Big\},$$ we consider the family of solutions given by
$$\lambda(x,y)=\log(a\cosh x + b\cosh y)$$
which are represented by 
$g(z)=-a\,\mathsf{exp}(z)-b,$ $z \in \Omega$.
We have that $g'(z)\neq 0$ and
$$1 + g\overline{g} = a^2(e^{2x}+1) + 2abe^x\cosh y> 0.$$
Moreover, using \eqref{f}, $f(z)=-\mathsf{exp}(-z)/a$ and
the Weierstrass representation~\eqref{RW} gives the following immersion 
$$\psi(x,y)= \Big(x,\Big(\frac{e^{-x}}{a}+a\sinh x\Big)\cosh y, -a\cosh x\sinh y\Big)+b\,\Big(-\frac{e^{-x}}{a}\cosh y,x,-y\Big).$$
We observe that $a=1$ and $b=0$ yields the {\em timelike hyperbolic catenoid of 1st kind}.

By Proposition~\ref{prop5} the lines of curvature are the coordinate lines. We have to show that $\alpha(x)=\psi(x,y_0)$ and  $\beta(y)=\psi(x_0,y)$ are plane curves. 

Let $\rho(x,y)= a\cosh(x) + b\cosh(y)$. Since $\rho_{xy}=0$, then the lines of curvature are plane curves and the immersion represents a timelike Bonnet surface (see Lemma~2.8 and Theorem~2.23 at \cite{Akamine}).

\begin{figure}[ht]\label{fig-BT}
\begin{center}
\subfigure{\includegraphics[width=0.28\textwidth]{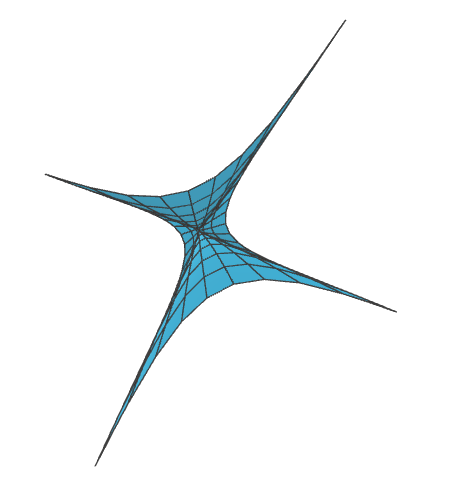}}
\hspace{0.5cm}
\subfigure{\includegraphics[width=0.27\textwidth]{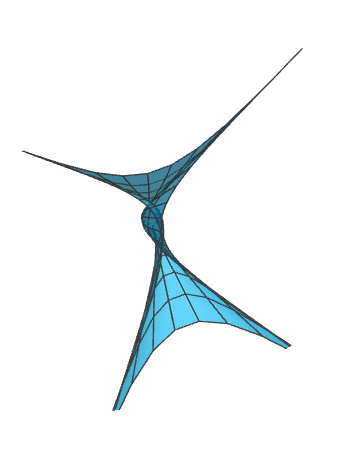}}
\hspace{0.5cm}
\subfigure{\includegraphics[width=0.25\textwidth]{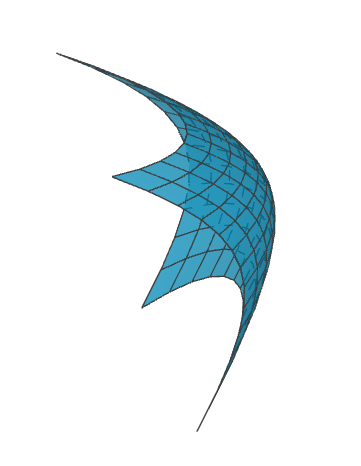}}
\end{center}\caption{Timelike Bonnet surfaces obtained for $a=3$, $a=2$ and $a=1$ (hyperbolic catenoid of 1st kind), respectively.}
\end{figure}

\end{example}

\begin{remark}
Two timelike minimal immersions $\psi, \psi^*:\Omega\subset\L\to \mathbb{L}^3$ are called {\em Lorentz-conjugate} (see \cite{Alexandre}) when
$$\left\{\begin{aligned}
\psi_u &= \psi^*_v,\\
\psi_v &= \psi^*_u. 
\end{aligned}\right.$$  
It is easy to prove that if $\psi$ has Gaussian curvature $K<0$, then the  Gaussian curvature of $\psi^*$ is  $K > 0$. Consequently, the Weingarten map of $\psi^*$ is not diagonalizable and we cannot apply Proposition~\ref{prop2}.
\end{remark}

\end{document}